\documentclass[11pt]{article}
\title{Models of weak theories of truth}

\usepackage{amsthm}
\usepackage{amssymb}
\usepackage{amscd}

\usepackage{amsmath}
\usepackage{amsfonts}
\usepackage{fancyhdr}
\usepackage{comment}

\newtheorem{thm}{Theorem}[section]
\newtheorem{fak}[thm]{Fact}
\newtheorem{cor}[thm]{Corollary}
\newtheorem{lem}[thm]{Lemma}
\newtheorem{stw}[thm]{Proposition}
\theoremstyle{definition}
\newtheorem{defin}[thm]{Definition}

\newtheorem{exa}[thm]{Example}

\newtheorem{obs}[thm]{Observation}
\newtheorem{konw}[thm]{Convention}

\newcommand{\df}[1]{\textbf{#1}}

\newcommand{\num}[1]{\underline{#1}}
\newcommand{\CT}{\textnormal{CT}}
\newcommand{\TB}{\textnormal{TB}}
\newcommand{\UTB}{\textnormal{UTB}}

\newcommand{\PA}{\textnormal{PA}}
\newcommand{\PAP}{\textnormal{PAP}}
\newcommand{\PAT}{\textnormal{PAT}}
\newcommand{\ZFC}{\textnormal{ZFC}}
\newcommand{\Term}{\textnormal{Term}}
\newcommand{\Sent}{\textnormal{Sent}}
\newcommand{\Form}{\textnormal{Form}}
\newcommand{\Clterm}{\textnormal{ClTerm}}
\newcommand{\ClTerm}{\Clterm}

\usepackage[utf8]{inputenc}
\usepackage{lmodern}
\author{Mateusz Łełyk, Bartosz Wcisło}

\begin{document}

\maketitle

\begin{abstract}
In the following paper we propose a model-theoretical way of comparing the "strength" of various truth theories which are conservative over $\PA$.
 Let $\mathfrak{Th}$ denote the class of models of $\PA$ which admit an expansion to a model of theory $Th$. We show (combining some well known results and  original ideas) that 
$$\mathfrak{PA}\supset \mathfrak{TB}\supset \mathfrak{RS}\supset \mathfrak{UTB}\supseteq\mathfrak{CT^-},$$
where $\mathfrak{PA}$ denotes simply the class of all models of $\PA$ and $\mathfrak{RS}$ denotes the class of recursively saturated models of $\PA$. Our main original result is that every model of $\PA$ which admits an expansion to a model of $\CT^-$, admits also an expanion to a model of $\UTB$. Moreover, as a corollary to one of the results we conclude that $\UTB$ is not relatively interpretable in $\TB$, thus answering the question from \cite{fujimoto}.

\end{abstract}

\section{Introduction}

Our paper concerns models of weak theories of truth. By a ''theory of truth'' we mean an extension of Peano Arithmetic (henceforth denoted by $\PA$) with axioms for an additional unary predicate $T(x)$ with intended reading ''$x$ is (a G\"{o}del code of) a true sentence''. We call a theory of truth ''weak'' iff it is a \emph{conservative} extension of $\PA$. Theories of truth are of interest both because of the philosophical context in which they emerged and of insights into the structure of models of $\PA$ which they provide us with. Let us briefly comment on both these issues.

Truth theories constitute a well established area of research in contemporary epistemology and logic (for a comprehensive introduction and reference see \cite{Halb}). In particular weak theories of truth have been introduced as a coherent formal framework for an explication of some stances in the debate over the metaphysical status of the notion of truth, specifically so called deflationary theories of truth. The deflationists claim that:

\begin{enumerate}
	\item Sentences of the form ''$\phi$ is true'' do not ascribe any actual property to the sentence 	$\phi$,
	\item The meaning of truth predicate is completely analysable in terms of Tarski's disquotation 			scheme.
\end{enumerate}  

 The first claim is contemporarily explicated by some authors (see \cite{Shap}, \cite{Field}) in terms of conservativity of a theory of truth over a theory of syntax (the latter usually modelled as $\PA$). Namely: the claim that the predicate ''$\phi$ is true'' does not express any actual property is rearticulated as a thesis that the correct theory of truth should be conservative over $\PA.$ This is precisely where the interest in weak theories of truth as defined in the following article comes from. 
 
The second claim is usually explicated as a thesis that the notion of truth is axiomatizable by Tarski's scheme (i.e. axioms of the form $T \ulcorner \phi \urcorner \equiv \phi,$ where $\phi$ is a formula) or, more precisely, some syntactic restriction thereof (e.g. Tarski's scheme restricted to arithmetical sentences; for a detailed discussion of this explication of the deflationary theory of truth see \cite{Ketl}, for a discussion of both theses see again \cite{Halb}). Formal theories of truth satisfying both above conditions such as $\TB$ and $\UTB$, defined later on in this paper, are subject to investigation as deflationary theories \textit{par excellence}.

Our research in the structure of models of weak theories of truth is related to another possible interpretation of the first claim of the deflationary theory of truth. This explication claims that the correct theory of truth should be \emph{model-theoretically conservative} over $\PA,$ i.e. every model $M$ of $\PA$ should admit an expansion to a model $(M,T)$ of the deflationary theory of truth. Speaking a bit na\"ively, if ''$\phi$ is true'' really doesn't express any genuine property, then its admissibility should not impose any conditions on how the object domain looks like, which can in turn be directly explicated as a model-theoretic conservativeness of the theory of truth. A very similar argument for the syntactic conservativeness interpretation of the deflationary theory of truth has been presented in \cite{Shap}.

As of this moment, we are highly sceptical towards the adequacy of this explication in the debate on deflationism. Even if it is the case that some weak theory of truth is not model-theoretic conservative over $\PA,$ this is not a substantial objection to a deflationist, who might be sceptical \emph{exactly} whether Tarskian semantics provides a correct analysis of the relationship between the actual language and its object domain. It might happen however, that in course of the philosophical debate new arguments emerge for this stronger notion of conservativeness as an adequate explication. 

Model theoretic considerations might be also seen as a tool for fine-grained classification of the weak theories of truth whose strength cannot be measured by merely proof-theoretical considerations, since most important ones, $\UTB$ and $\CT^-$, are incomparable. Namely: $Th_1$ could be deemed not stronger than $Th_2$ if all models of $\PA$ that admit an expansion to $Th_2,$ admit also an expansion to $Th_1.$ In the following paper we actually prove that for the three most important weak theories of truth, i.e. $\TB$, $\UTB$ and $\CT^-$ the classes of models of $\PA$ that admit an extension to the models of respective truth theory can be linearly ordered by inclusion. I.e our result is that: 
\begin{equation}\label{equat: main}\tag{$\ast$}
\mathfrak{PA}\supset \mathfrak{TB}\supset \mathfrak{RS}\supset \mathfrak{UTB}\supseteq\mathfrak{CT^-}
\end{equation}
where $\mathfrak{PA}$ denotes simply the class of all models of $\PA,$ $\mathfrak{RS}$ denotes the class of recursively saturated models of $\PA$ and by $\mathfrak{TB}, \mathfrak{UTB}, \mathfrak{CT}^-$ we mean the class of models of $\PA$ which admit an extension to a model of $\TB, \UTB, \CT^-$ respectively. Note that $\supset$ means ''strict inclusion''. Note that we \emph{do not assume} that models of $\PA$ we deal with are countable, which would make the right part of the above sequence trivially collapse, due to Barwise-Schlipf theorem (i.e. Fact \ref{bsr} in this paper).  

Seen from the purely model-theoretical perspective weak theories of truth are handy tool for obtaining interesting results about the structure of models of $\PA$. Most striking examples of their implementation include: easy proof of Smorynski--Stavi theorem, proof in $\ZFC$ of the existence of recursively saturated rather classless models of $\PA$ (both due to Schmerl, see \cite{Schmerl}) and the result that countable recursively saturated models of $\PA$ have recursively saturated end extensions. Although all these proofs had been established independently, weak theories of truth provided conceptual and uniform way of dealing with those complicated structures.

\subsection{A Short Commentary on the Main Result}

Let us have a word of comment on the main result. We proved that every model of $\PA$ which admits an expansion to a model of $\CT^-$, admits also an expansion to $\UTB$. Note that this is a strengthening of Stuart Smith's result given in \cite{Smith} as we show that the undefinable class proven to exist by Smith can be fully inductive. Moreover one can deduce Lachlan's theorem ("models of $\CT^-$ are recursively saturated") from it since an easy argument from overspill demonstrates that models of $\UTB$ are recursively saturated.\\

When one restricts attention to \emph{countable} models of $\PA$ then our result is an easy consequence of classical results in the model theory for $\PA$. What is more, a stronger theorem is true: both theories \emph{have exactly the same} countable models, and these are precisely recursively saturated ones. Let us sketch the argument: by Lachlan theorem (see \cite{Kaye}) every model of $\CT^-$ is recursively saturated and by Kotlarski-Krajewski-Lachlan Theorem (\cite{Kaye}) every such model carries an interpretation of $\CT^-$\footnote{One can give independent proof of this fact using resplendency of countable recursively saturated models (which is the content of Barwise-Schlipf-Ressayre Theorem (\ref{bsr})) and conservativity of $\CT^-$ proved independently by Enayat-Visser (in \cite{Enayat}) and Leigh (in \cite{Leigh})}. This ends the argument for $\CT^-$. To prove that each countable recursively saturated model of $\PA$ carries an interpretation of $\UTB$ truth class one uses resplendency of countable and recursively saturated models and proof-theoretical conservativity of $\UTB$ over $\PA$, which is a folklore result (which can be found in \cite{Halb}). The converse direction uses induction for the truth predicate in a straightforward way, and we give a proof of it in Section \ref{easy} (Proposition \ref{UTB to RS}).\\

In the general case things get much more complicated as many intriguing models emerge. First observation is that in a sense there are strictly more recursively saturated models than models which admit an interpretation of $\CT^-$ or $\UTB$. This is a consequence of the fact that rather classles models of $\PA$ can be recursively saturated (as stated in Theorem \ref{Schmerlthm} proved by Kauffman) and the fact that models of $\CT^-$ and $\UTB$ always carry an undefinable class (for $\CT^-$this is the content of Stuart Smith's theorem given in \cite{Smith}\footnote{Note that this result also becomes trivial when restricted to countable case}. For $\UTB$ this trivially follows by Tarski's theorem, as the interpretation of $\UTB$-truth predicate, being fully inductive, is always a class.) Both results are highly non-trivial. It is worth emphasizing that rather classles models of $\PA$ cannot separate $\UTB$ from $\CT^-$, hence one needs different tools to find a structure (if it can be found) which admits an interpretation of the latter but not the former theory. Our main result shows that as far as the inclusion $\mathfrak{CT}^-\subseteq \mathfrak{UTB}$ is concerned things behave like in the countable case, but this requires substantially different arguments.

\subsection{Structure of the paper}

In the Section \ref{sect: not and def} we introduce all the notions, both from truth theory and model theory of $\PA$, necessary to understand our results. In the next section (i.e. Section \ref{easy}) we give construction of a model of $\TB$ which cannot be extended to a model of $\UTB$, which is our original result. In the rest of this section we prove some well-known results which are needed to have our tower of inclusions (depicted in \eqref{equat: main}) completed. We devote Section \ref{sect: main} to the proof of our main theorem. This is the most technical and definitely the most difficult part of our paper. In the appendix, for the sake of completeness, we give Schmerl's construction of rather classles, recursively saturated model of $\PA$.

\section{Notation and definitions}\label{sect: not and def}

In this section we would like to introduce key definitions along with some notation. As for the latter we make a number of simplifications which strictly speaking might be ambiguous but no reasonable reading may cause any confusion. Considerations of fairly logical nature might be easily obscured by inappropriately heavy coding and putting too much stress on this aspect, which we tried to avoid.

\begin{konw}
	\begin{enumerate}
	\item $\PA$ denotes Peano Arithmetic, and $\mathcal{L}_{\PA}$ is the language in which it is formalized 			(for the sake of definiteness we assume that $\mathcal{L}_{\PA} = \{\cdot, +, \leq, 0, 1, S\}$ where $S$ is a one argument function and $\cdot$, $+$ are two argument functions).
	\item We use big capital letters $M$, $N \ldots$  for models of $\PA$ even if not stated explicitly.
	\item We use $\Form(x)$, $\Sent(x)$, $\Term(x)$, $\ClTerm(x)$ to denote formulae representing in $\PA$ sets of (G\"odel codes of) respectively (arithmetical) formulae, sentences, terms and closed terms. If $M$ is a model of $\PA$, then by $\Form(M)$ we mean the set of (the codes of) arithmetical formulae in this model. Similarly for $\Sent, \Term$ and $\Clterm$.  We use $\num{x}$ to denote the $x$-th numeral, i.e. the closed term of the form $S \ldots S (0),$ where the successor symbol $S$ has been repeated $x$ times. 
	\item We skip Quine's corners when talking about G\"odel codes of syntactic objects, e.g. we write
	$$\Phi(\psi)$$
	instead of $\Phi(\ulcorner \psi\urcorner).$
	\item 
	We will implicitly assume that, when bounded by a quantifier, variables $s,t, \ldots$ refer to (G\"odel codes of) terms and $\phi, 		\psi, \ldots$ refer to (G\"odel codes of) formulae. In particular we write 
	$$\forall t \ \ \phi(t)$$
	instead of
	$$\forall x (\ClTerm(x) \longrightarrow \phi(x)),$$
	and we treat
	$$\forall  \psi \ \ \Phi(\psi)$$
	in the same fashion. Analogously for the existential quantifier.
		\item We write $t^{\circ}$ to denote the result of formally evaluating the (G\"odel code of) term $t$. If $\bar{t} = (t_1, \ldots, t_n)$ is a tuple of terms of a fixed length, then $\bar{t}^{\circ} = (t_1^{\circ}, \ldots, t_n^{\circ}).$
		\item We sometimes write the result of syntactical operations with no mention of the operations 			themselves e.g.
	$$\exists t \ \ \Psi(\phi(t))$$
	stands for
	$$\exists t \ \ \Psi(Subst(\phi, t))$$
	where $Subst(x,y)$ is a formula representing substitution function. In a similar fashion
	$$\Xi(\phi \wedge \psi)$$
	stands for
	$$\forall \theta \ \ \Bigl( \theta=Conj(\phi, \psi) \rightarrow \Xi(\theta) \Bigr).$$
where $Conj(x,y)$ stands for the formula representing function which takes two (G\"odel codes of) formulae to (the G\"odel code of) their conjunction.
\item If $P(x)$ is any predicate, then we denote the language $\mathcal{L}_{\PA}\cup \{P\}$ by $\mathcal{L}_{\PAP}$.
\item If $P(x)$ is any predicate then by $Ind(P)$ we mean the set of all instantiations of induction scheme for all formulae in the language $\mathcal{L}_{\PAP}$. If $\phi$ is any formula then by $Ind ( \phi)$ we mean the induction axiom for the formula $\phi.$ 
\end{enumerate}
\end{konw} 

We shall now introduce the theories which will be considered in this paper. They have been all extensively discussed in \cite{Halb}.

\begin{defin} All the theories are formalised in the language $\mathcal{L}_{\PAT}$ and are extensions of $\PA$ (below we list only the additional axioms).
\begin{enumerate}
\item $\TB^-$ is a theory axiomatized by the scheme (called \textit{Tarski Biconditional} scheme) 
$$T\phi \equiv \phi,$$
where $\phi$ is a sentence of $\mathcal{L}_{\PA}.$
\item $\UTB^-$ is a theory axiomatized by the scheme (called \textit{Uniform Tarski Biconditional} scheme)
$$\forall \bar{t} \bigl(T\phi(\bar{t})\equiv\phi(\bar{t}^{\circ})\bigr),$$
where $\phi$ is a formula of $\mathcal{L}_{\PA}$
\item $\CT^-$ is finitely axiomatized by the following sentences
	\begin{enumerate}
		\item $\forall t,s  \ \ \Bigl(T(R(t,s)) \equiv R(t^{\circ},s^{\circ})\Bigr)$ where $R$ is $=$ or $\leq$.
		\item $\forall  \phi,\psi \ \ \Bigl(T(\phi\otimes \psi)\equiv T(\phi)\otimes T(\psi) \Bigr)$, where $\otimes$ is $\wedge$ or $\vee$.
		\item $\forall \phi \ \ \Bigl(T(\neg\phi)\equiv \neg T(\phi) \Bigr)$.
		\item $\forall \phi \ \ \Bigl(T(Qx\phi(x))\equiv Qt T(\phi(t))\Bigr),$ where $Q$ is $\exists$ or $\forall$.
		
	\end{enumerate}
\item $\TB$, $\UTB$, $CT$ are the extensions of $\TB^-$, $\UTB^-$, $\CT^-$ respectively with full induction for the enriched language, i.e.
$\TB$ = $\TB^- \cup Ind(T)$, $\UTB = \UTB^-\cup Ind(T)$, $CT = \CT^- \cup Ind(T)$. 
\end{enumerate}
\end{defin}

\begin{konw}
As suggested by the examples of weak theories of truth in the Introduction, if $Th$ is any theory extending $\PA$ then by $\mathfrak{TH}$ we denote the class of those models of $\PA$ which admits an extension to a model of $T$.
\end{konw}

Let us state some standard results in model theory of $\PA$ which we will make use of. Their proofs may be found in \cite{Kaye}. In all the following facts we assume $\mathcal{L}$ to be any language extending $\mathcal{L}_{\PA}.$

	\begin{fak}[Overspill lemma] \label{overspill} 
	Let $\phi(x) \in \mathcal{L}$ be any formula. Let $M$ be any $\mathcal{L}$-structure such that $M \models Ind (\phi).$ Suppose that for all $n \in \omega$, $M \models \phi(n).$ Then for some nonstandard $c \in M$ 
	$$M \models \phi(c).$$
	\end{fak}	
	
	\begin{fak}[Prime models] 
	Let $Th$ be any theory extending $\PA$ with full induction scheme for the whole $\mathcal{L}.$  Then there exists the prime model $K(Th)$ of the theory $Th.$ Moreover, all elements of $K(Th)$ are definable, i.e. for any $c \in K(Th)$ there exists a formula $\phi(x) \in \mathcal{L}$ such that
	$$K(Th) \models \phi(c) \wedge \exists ! \ x \phi(x).$$
	\end{fak}
	
	Let us define three important kinds of extensions of models of $\PA.$ 
	
	\begin{defin}
	Let $M \subset M'$ be any extension of models of $\PA.$ We call such ans extension \emph{conservative} (denoted $\subset_{cons}$) iff for any set $X \subset M'$ definable in $M'$ the set $X \cap M$ is definable in $M.$ We call it an \emph{end extension} (denoted $\subset_{end}$) iff any $c \in M' \setminus M$ dominates $M,$ i.e. $c>b$ for all $b \in M$ and $M' \setminus M \neq \emptyset.$ We call an extension \emph{cofinal} (denoted $\subset_{cf}$) iff for any $c \in M'$ there exists $b \in M$ such that $c<b.$ If such an extension is additionally elementary we denote it $\prec_{cons}, \prec_{end}, \prec_{cf}$ respectively.
	\end{defin}	
	
	It is an easy fact that any conservative extension is actually an end extension. Less obviously we have the following theorem.
	
	\begin{fak}[MacDowell-Specker Theorem] \label{mds}
	Let $Th$ be any theory extending $\PA$ with full induction scheme for $\mathcal{L}$ and let $M \models Th.$ Then there exists a model $M \prec M' \models Th$ such that $M'$ is a conservative extension of $M.$
	\end{fak}

	\begin{defin}
	Let $M$ be any $\mathcal{L}$-structure. Let $X \subset M.$ We call $X$ a \emph{(proper) class} iff $X$ is not definable in $M$ (i.e. is not definable with parameters) but for any $a \in M$ the set $\{ x \in X \ | \ x<a \}$ is definable in $M.$
	\end{defin}

	\begin{fak}[Barwise-Schlipf-Ressayre] \label{bsr}
	Let $M$ be any recursively saturated countable model of a recursive theory $Th$ extending $\PA.$ Then for any finite tuple $\bar{a}\in M$, any recursive theory $Th'$ in a recursive language $\mathcal{L}'$ extending $\mathcal{L}\cup\{\bar{a}\}$ if $Th' + Th(M, \bar{a})$ is consistent, then $M$ admits an expansion to a model $M'$ of $Th'.$ The models $M$ with the above property for an arbitrary single sentence in place of recursive theory are called \textit{resplendent}.
	\end{fak}
	
\section{Easy or classical results}\label{easy}

In this section we show how to prove almost all inclusions mentioned in the introduction. Let us begin with some trivial observations:

\begin{stw}
$\mathfrak{TB}\supseteq \mathfrak{UTB}$
\end{stw}
\begin{proof}
This follows immediately from the fact, that $\TB$ is a subtheory of $\UTB$.
\end{proof}

\begin{fak}\label{UTB to RS}
$\mathfrak{RS}\supseteq \mathfrak{UTB}$
\end{fak}
\begin{proof}
Fix any model $(M,T)\models \UTB$ and a recursive type $p(x, \bar{a})$ over $M$ consisting of arithmetic formulae with parameters $a_1, \ldots, a_n \in M$. Let $\phi(z)$ represent $p(x,y_1, \ldots, y_n)$ in $\PA$. Since $p(x,a_1, \ldots, a_n)$ is finitely satisfied in $M$, for all $k\in \mathbb{N}$
$$(M,T) \models \exists c  \forall \psi<k \ \  \Bigl( \phi(\psi) \longrightarrow T \psi(\num{c}, \num{a_1}, \ldots, \num{a_n} )\Bigr).$$
Hence, by overspill, there is a nonstandard $b\in M$ and a $c\in M$ s.t.
$$M \models \forall \psi<b  \ \ \Bigl( \phi(\psi) \longrightarrow T \psi(\num{c}, \num{a_1}, \ldots, \num{a_n}) \Bigr).$$
Which proves the thesis.
\end{proof}

In proving the inclusion $\mathfrak{TB}\supset \mathfrak {RS}$, we will need the unpublished characterization of $\mathfrak{TB}$, which has been found independently by Fredrik Engstr\"{o}m and Cezary Cieśliński. 

\begin{stw}\label{Engs-Ciesl}
$M \in \mathfrak{TB}$ if and only if the set
$$Th_{\mathcal{L}_{\PA}}(M) = \{\phi \in \mathbb{N} \ \ | \ \ \phi \in \Sent_{\mathcal{L}_{\PA}} \textrm{ and } M \models \phi\}$$
is coded in $M$.
\end{stw}

\begin{proof} Fix any model $M \models \TB$.\\
($\Rightarrow$) Observe that for all $n\in \mathbb{N}$, $$(M, T)\models  \exists x \forall \phi<n \ \ \bigl(\phi \in x \equiv T\phi\bigr)$$
and use overspill to find a code of the theory of $M$.\\
($\Leftarrow$) Take $T = \{a\in M\ \ | \ \ M\models a\in c\}$, where $c$ is the code of $Th_{\mathcal{L}_{\PA}}(M)$.
%Let us occupy with the left-to-right implication first. Suppose that $\mathcal{M}\in \mathfrak{TB}$, and pick $T\subseteq M$ such that
%$$(\mathcal{M},T)\models TB.$$ 
%Since finite subsets are coded already in $\mathbb{N}$ we get that for $n\in \mathbb{N}$
%$$(\mathcal{M}, T)\models \exists x \forall \phi <n \bigl(\phi \in x \equiv T\phi\bigr)$$
%Now by overspill it follows that there is a $c\in M\setminus \omega$ such that
%$$(\mathcal{M}, T)\models \exists x \forall \phi <c \bigl(\phi \in x \equiv T\phi\bigr)$$
%Hence there is some $d \in M$ satisfying $(\mathcal{M}, T)\models \forall \phi<c \bigl(\phi\in d \equiv T\phi\bigr)$. By al axioms we can see that for all $\mathcal{L}_{PA}$ sentences
%$$\mathcal{M}\models \phi \in d \iff (\mathcal{M}, T)\models T\phi \iff\mathcal{M}\models \phi.$$
%Hence $d$ codes all the $\mathcal{L}_{PA}-$sentences that are true in $\mathcal{M}$.\\
%Right-to-left direction is even more elementary. Fix any model $\mathcal{M}$ of $PA$. Suppose that $c\in M$ is a code of $Th_{\mathcal{L}_{PA}}(\mathcal{M})$. Now we can define the extension of the predicate $T$ in $M$ as simply
%$$\{a\in M \ \ | \ \ \mathcal{M}\models a\in c\}.$$
%Now disquotational axioms are true in $(\mathcal{M}, T)$ because $c$ is a code of a theory of $\mathcal{M}$. Induction axioms for $T$ follows because in fact $T$ is definable with parameters in $\mathcal{M}$.
\end{proof}

\begin{cor}
$\mathfrak{TB}\supseteq \mathfrak{RS}$ 
\end{cor}
\begin{proof}
Suppose that $M$ is recursively saturated and consider the following recursive type with a free variable $x$:
$$\{\phi \equiv \phi\in x \ \ | \ \ \phi \in \Sent_{\mathcal{L}_{\PA}} \}.$$
Any element of $M$ realizing this type will be a code of the theory of $M$. Hence by Proposition $\ref{Engs-Ciesl}$ the model $M$ can be extended to a model of $\TB$.
\end{proof}

\begin{cor}[Cieśliński, Engstr\"om]
$\mathfrak{TB} \subset \mathfrak{PA}.$
\end{cor}

	\begin{proof}
		Take any prime model $K$ of a complete extension $Th \neq Th(\mathbb{N})$ of $\PA.$ Suppose that $K$ admits an expansion to a model 
		$$(K,T) \models \TB.$$ 
But then by Proposition \ref{Engs-Ciesl} there would be an element $c \in K$ which codes the theory $Th.$  Since all elements of $K$ are arithmetically definable without parameters, the formula
		$$x \in c$$
would then yield an arithmetical definition of truth for $Th,$ contradicting Tarski's theorem.		
	\end{proof}
	
Now we proceed to the construction of a model which codes its theory and is not recursively saturated, in this way proving that the inclusion $\mathfrak{TB}\supseteq \mathfrak{RS}$ is strict. It shows up that this is an easy consequence of MacDowell-Specker theorem and the following lemma:

\begin{lem}\label{nsr}
 Suppose that  $M \prec_{cons} M'$ i.e. $M'$ is an elementary, conservative end extension of $M$. Then $M'$ is not recursively saturated.
\end{lem}
\begin{proof}
Let $M$, $M'$ be as in the formulation of the lemma and suppose that $M'$ is recursively saturated. Pick $c\in M'\setminus M$ and let $b\in M'$ realize the following recursive type with the free variable $y$:
$$\{\forall \bar{x}  \ \ \Bigl(\phi(\bar{x}) < c \longrightarrow (  \phi(\bar{x}) \equiv (\phi(\num{\bar{x}}) \in y) \Bigr) \ \ | \ \ \phi(\bar{z}) \in \Form_{\mathcal{L}_{\PA}} \}.$$

Consider the set
$$ X= \{ a \in M \ | \ M' \models a \in b \}.$$
Since the extension $M \prec M'$ is conservative, $X$ should be definable with parameters from $M.$ Then by definition of $b$ the elementary diagram of $M$ would also be definable in $M$, which contradicts Tarski's Theorem.

\end{proof}

\begin{thm} \label{tb a rs}
	$\mathfrak{TB} \supset \mathfrak{RS}.$ Moreover, every model $M$ has an elementary extension to $(M',T) \models \TB$ with $M'$ not recursively saturated.
\end{thm}

\begin{proof}
We prove the ''moreover'' part which of course suffices. Let us fix any $M$. Let $c$ be a fresh constant. By compactness the following theory
$$ElDiag(M)\cup \{\phi \in c \ \ | \ \ \phi \in \Sent_{\mathcal{L}_{\PA}} \wedge M\models \phi\}$$
has a model $M',$ which is an elementary extension of $M$. Note that $c$ is a code of the theory of $Th_{\mathcal{L}_{\PA}}(M')$. Using MacDowell-Specker theorem (Theorem \ref{mds}) we can find 
$$M'' \succ_{cons} M'.$$
 Since $Th(M') = Th(M''),$ we see that $c\in M''$ is a code of a theory of $M''$. By Proposition \ref{Engs-Ciesl} $M''$ can be expanded to a model $(M'',T) \models \TB$. But by lemma \ref{nsr} it cannot be recursively saturated.
\end{proof}

As a corollary we obtain a solution to a problem from \cite{fujimoto}. 

\begin{cor} [Nicolai]
$\UTB$ is not relatively interpretable in $\TB,$ i.e. there is no formula $\phi(x)\in\mathcal{L}_{\PAT}$ defining in every model $(M,T) \models \TB$ a subset $S$ such that $(M,S) \models \UTB.$ 
\end{cor}

When we proved Theorem \ref{tb a rs} we were not at all aware of the problem. The above corollary has been formulated by Carlo Nicolai, who has also brought Fujimoto's paper to our attention.

Let us now return to the inclusion $\mathfrak{RS}\supseteq \mathfrak{UTB}$. It is easy to show that for a counterexample to equality here we have to search among models with \textit{uncountable cofinality}. By a result of Smorynski--Stavi (see \cite{Schmerl}) if $T$ is an extension of $\PA$ in a language $\mathcal{L}\supseteq\mathcal{L}_{\PA}$ such that $T$ contains induction axioms for $\Form_{\mathcal{L}}$, then $T$ is preserved in cofinal extensions, i.e.
$$M \models T\wedge M \prec_{cf} N \Rightarrow N \models T.$$
Putting it together with the fact that countable and recursively saturated models of $\PA$ are resplendent and that $\UTB$ is a conservative extension of $\PA$ we see that every recursively saturated model of $\PA$ with \textit{countable cofinality} can be expanded to a model of $\UTB$. In order to prove the existence of recursively saturated models which do not expand to a model of $\UTB$ we will profit from the fact that the interpretation of $\UTB$-truth predicate, if exists, is always a proper class.
\begin{obs}
If $T\subseteq M$ is such that $(M, T)\in \mathfrak{UTB}$ then $T$ is a proper class on $M$. Indeed, $T$ is a class, because $\UTB$ contains induction axioms for all formulae of enriched language and $T$ is obviously undefinable by Tarski's theorem.
\end{obs}

Recall that model $M$ is \textit{rather classless} if it contains no proper class. Interestingly, the following theorem holds:

\begin{thm}\label{Schmerlthm}
There exists a recursively saturated and rather classless model of $\PA$.
\end{thm}

The existence of recursively saturated rather classless models of $\PA$ was first demonstrated by Matt Kauffman in $\ZFC + \diamond$ (\cite{Kauf}). The assumption about existence of $\diamond$-sequence was later eliminated by Shelah (in \cite{She}). In the appendix we will present another argument for the existence of recursively saturated rather classless models which was given in \cite{Schmerl}.

\section{The Main Result}\label{sect: main}	
		
		In the following part we will present the most technically involved part of our result i.e.
		
	\begin{thm}
	$\mathfrak{CT}^- \subseteq \mathfrak{UTB},$ i.e. for any $(M,T) \models \CT^-,$ there exists $T'$ such that $(M,T') \models \UTB.$	
	\end{thm}				
	
		We will precede the proof of the theorem with two lemmata. The first  of them states that in any nonstandard model $(M,T)$ of $\CT^-$ we can find nonstandard \emph{arithmetical} truth predicates which, in a sense, satisfy $\UTB^-$, i.e. there exists a formula
		\begin{displaymath}
		T'(x) := T \rho(\num{x}),
		\end{displaymath}
		 where $\rho(x) \in \Form(M)$ such that for an arbitrary standard arithmetical formula $\psi$ and arbitrary terms $t_1, \ldots, t_n \in \Term(M)$ we have
		\begin{displaymath}
		(M,T) \models T' \psi(t_1, \ldots, t_n) \equiv \psi(t_1^{\circ}, \ldots, t_n^{\circ}).
		\end{displaymath}
		Let us say that in such a case the formula $T'$ satisfies the \df{Uniform Disquotation Scheme}. This lemma was actually already present in the paper \cite{Smith}. 	Let $(\chi_i)$ be an arbitrary primitive recursive enumeration of arithmetical formulae. We will repeatedly refer to this enumeration.
		
			\begin{defin}
				By $\phi[ \xi \mapsto \delta]$ we mean the result of formally substituting a formula $\delta$ for every occurrence of the boolean subformula $\xi$ in a formula $\phi.$ 	
			\end{defin}

		\begin{lem} \label{lem_UTB_family}
			There exists a primitive recursive family of formulae $(\rho_n)$ such that any for any $n \in \omega$ and $i\leq n$ provably in $\CT^-$ we have
			\begin{displaymath}
			\forall \bar{t} \ \Bigl( T \rho_n(\num{\chi_i(\bar{t})}) \equiv \chi_i(\bar{t^{\circ}}) \Bigr).
			\end{displaymath}
			Moreover, for an arbitrary model $(M,T) \models \CT^-$, for an arbitrary nonstandard $a$ and for an arbitrary $i < \omega$ (i.e. for an arbitrary standard $i$) we have
			\begin{displaymath}
			\forall \bar{t} \ \Bigl( T \rho_a(\num{\chi_i(\bar{t})}) \equiv \chi_i(\bar{t^{\circ}}) \Bigr).
			\end{displaymath}
		\end{lem}
		
		In other words, the above lemma states that there exists a primitive recursive family of arithmetical formulae which behave like partial truth predicates for indefinitely growing finite sets of standard formulae ($n$-th truth predicate works fine for first $n$ formulae in our enumeration) such that nonstandard truth predicates from this family behave well for all standard formulae.
		
		\begin{proof}
			We will try to construct an analogue of simple arithmetical partial truth predicates 
			$$ \tau(x) = \bigvee_{i=1}^n ( x=\phi_i ) \wedge \phi_i$$
			but in such a way that they behave well in all models of $\CT^-$ also for nonstandard $n$. Let $(\chi_i)_{i< \omega}$ be an our fixed primitive recursive enumeration of arithmetical formulae. Let
			\begin{eqnarray*}
				\xi_{2i}(x) & = & \forall \bar{t}  \bigwedge_{ j \leq i}     \ \Bigl(  x= \chi_j(\bar{t})  \longrightarrow \chi_j(\bar{t}^{\circ}) \Bigr) \cr
				\xi_{2i+1}(x) & = &  \exists \bar{t} \bigvee_{j \leq i}  \ \Bigl( x= \chi_j(\bar{t}) \Bigr)
				. \cr 
			\end{eqnarray*}
			
			Finally, let us define formulae which will play the role of the partial truth predicate $\tau$ above.
			
			\begin{eqnarray*}
				\rho_0 & = & \xi_0 \cr
				\rho_{2i+1} & = & \rho_{2i} [\xi_{2i} \mapsto \xi_{2i} \wedge \xi_{2i+1} ] \cr
				\rho_{2i+2} & = & \rho_{2i+1} [\xi_{2i+1} \mapsto \xi_{2i+1} \vee \xi_{2i+2} ]. \cr
			\end{eqnarray*}
			Obviously these definitions may be formalised in $\PA$.
			Note that a formula $\xi_i$ begins with a block of quantifiers. Thus although in principle a formula $\rho_j$ may have many subformulae of the form $\xi_j$, there is exactly one which is its \emph{boolean} subformula, namely the rightmost one. So, as $i$ gets bigger, the formulae $\rho_i$ keep growing only to the right.  Here are two simple properties of so defined $\rho_i$'s for arbitrary $i \in \omega$ and any nonstandard $a \in M$:

			\begin{enumerate}
				\item $T\rho_a(\num{x}) \rightarrow T\rho_{2i}(\num{x}).$
				\item $T\rho_{2i+1}(\num{x}) \rightarrow T \rho_a(\num{x}).$				
			\end{enumerate}
			
			To prove the first one assume that $T \rho_a(\num{x})$ holds. Note that 
			$$ \rho_a(\num{x}) = \rho_{2i}[ \xi_{2i} \mapsto \xi_{2i} \wedge \gamma ](\num{x})$$
			for some nonstandard formula $\gamma.$ Since $\rho_{2i}$ is a positive combination of the formulae $\xi_j$ (i.e. no negation symbol occurs in it as a boolean formula with $\xi_j$'s treated as propositional variables) a substitution of $\xi_{2i} \wedge \gamma$ for $\xi_{2i}$ yields a formula stronger (no weaker) than $\rho_{2i}.$ As it is purely a matter of finite boolean calculus, this can be proved in $\CT^-.$ Proof of the second implication is analogous.
			
			We are now in a position to show that for an arbitrary nonstandard $a$ the formula $ T \rho_a(x)$ satisfies the uniform disquotation scheme i.e for an arbitrary standard arithmetical $\phi$ and $t_1, \ldots, t_n \in \Term(M)$ the following equivalence holds: 
			$$ T \rho_a \num{\phi(\bar{t})} \equiv \phi( \bar{t} ^{\circ}).$$ 
			
			Let us take any $\phi(\bar{t}).$ We know that $\phi(\bar{t}) = \chi_i(\bar{t})$ for some $i \in \omega.$ Suppose 
			$$T \chi_i (\bar{t}).$$ 
			Then it is easy to see  that 
			\begin{enumerate}
				\item $T \xi_{2i+1}(\num{\chi_i(\bar{t})})$  
				\item $\neg T \xi_{2j+1}(\num{\chi_i(\bar{t})}),$ for $j<i$
				\item $T \xi_{2j}(\num{\chi_i(\bar{t})})$ for arbitrary $j.$ 
			\end{enumerate}		
			
			It is enough to show that $T \rho_{2i+1} (\num{\chi_i(\bar{t})}).$ To this end we will prove by external backwards induction that for any boolean subformula $\psi$ of $\rho_{2i+1}$ different from formulae of the form $\xi_l$ the righthandside part of $\psi$ is true (note that since $\xi_l$'s begin with a block of quantifiers, they are the minimal boolean subformulae of $\rho_{2i+1}$). By assumption and points 1-3. above the claim holds for the minimal  subformula of $\rho_{2i+1}$ different from $\xi_l$'s, i.e. for
			
			$$ \psi = \xi_{2i}(\num{\chi_i(\bar{t})}) \wedge \xi_{2i+1}(\num{\chi_i(\bar{t})}).$$
			
			Now any righthandside of any subformula of $\rho_{2i+1}$ is exactly of one of the two following forms:
			\begin{enumerate}
				\item $\xi_{2j} \wedge \gamma$
				\item $\xi_{2j+1} \vee \gamma.$
			\end{enumerate}
			But by induction hypothesis we can assume that $\gamma$ is true. Now, since the main connective in the formula $\rho_{2i+1}$ is a conjunction of the form $\xi_0 \wedge \gamma,$ we are able to prove in $\CT^-$ that 
			$$ T \rho_{2i+1}(\num{\chi_i(\bar{t})}).$$ By previous observation it follows that $$T \rho_a (\num{\chi_i(\bar{t})}).$$ The converse implication is handled in a similar fashion (but now using $T \rho_a(\num{x}) \rightarrow T \rho_{2i}(\num{x})$). Using similar arguments one can also show, that $\rho_{2n}$ satisfies uniform disquotation scheme for formulae $\chi_0,\ldots,\chi_n$ for arbitrary $n \in \omega.$
		\end{proof}

		Before we state the next lemma, let us introduce some notation. 
		
	\begin{defin}
		Let $\delta(\bar{x})$ be an arbitrary formula in the language $\mathcal{L}_{\PAP}$ i.e. the language of arithmetic with a unary predicate $P(x)$ added. Then for an arbitrary formula $\phi$ with one free variable by $\delta[\phi]$ we mean the result of formally substituting the formula $\phi(x_i)$ for any occurrence of $P(x_i)$ in the formula $\delta$ (possibly preceded by some fixed renaming of bounded variables in $\delta,$ so as to avoid clashes). If $M \models \delta[\phi]$ we will say that $\phi$ satisfies a property $\delta$ in the model $M$.
	\end{defin}

	Let us quickly give an example of the above notions, which probably could be more illuminating than a definition. 
	
	\begin{exa}
		Let $\delta(x,y) = \Bigl( P(x) \equiv P(y) \Bigr).$ Then 
		$$\delta[z=z] = \Bigl( (x=x) \equiv (y=y) \Bigr).$$
	\end{exa}
	
	Obviously both notions may be formalized in $\PA.$ We are now ready to state the main lemma. This is the combinatorial core of our theorem. Essentially it has been proved in \cite{Smith}, although for a special case. We reprove it for the convenience of the Reader. Basically, the lemma states that the existence of a truth predicate satisfying $\CT^-$ allows us to define a predicate satisfying $\UTB^-$ and some additional definable properties shared by arithmetical formulae. 
	
	\begin{lem}[Main Lemma] \label{lem_main}
	Let $\delta$ be an arbitrary formula in $\mathcal{L}_{\PAP}.$ Let $(M,T) \models \CT^-.$ Suppose that for an arbitrary standard arithmetical formula $\phi$ we have $$(M,T) \models \delta[\phi].$$ Then there exists a formula $T'(x)$ in $\mathcal{L}_{\PAT}$ with parameters from $M$ such that
$$(M,T) \models T' \psi(\bar{t}) \equiv \psi(\bar{t}^{\circ})$$ 	
	 for an arbitrary standard arithmetical formula $\psi$ and arbitrary terms $t_1, \ldots, t_n \in \Term(M)$ and moreover
$$(M,T) \models \delta[T'].$$
	\end{lem}
	
	Before we prove the lemma, we will state separately its technical core. To this end we need the following definition:
	
	\begin{defin}
		Let $(M,T) \models \CT^-$ be an arbitrary model. Recall that $\Form (M)$ denotes the set of arithmetical formulae with at most one free variable in the sense of the model $M.$ We define the \df{rank} function $r: \Form (M) \to \omega \cup \{- \infty, + \infty\} $ in the following way:
		\begin{displaymath}
		r(\gamma) = \left\{ \begin{array}{ll}
		- \infty, & \textnormal{if } (M,T) \models \neg T\delta[\gamma] \\
		+ \infty, & \textnormal{if } (M,T) \models T\delta[\gamma] \textnormal{ and for all $i \in \omega$ } \\
		& \phantom{if} (M,T) \models \forall \bar{t} \ \Bigl(T\gamma(\num{\chi_i(\bar{t})}) \equiv \chi_i(\bar{t}^{\circ}) \Bigr) \\
		n & \textnormal{if }  (M,T) \models T\delta[\gamma] \textnormal{ and $n+1$ is the least number } k \in \omega \\
		& \phantom{if} \textnormal{such that } (M,T) \models \neg \forall \bar{t} \ \Bigl(T\gamma(\num{\chi_k(\bar{t})}) \equiv \chi_k(\bar{t}^{\circ}) \Bigr).
		\end{array} \right.		
		\end{displaymath} 
	\end{defin}
	
	Although the above definition may seem overly technical, it is indeed very natural: the rank $r$ measures in somewhat na\"ive way, how close a given formula $\gamma$ come to satisfying Lemma \ref{lem_main}. The next technical result (which is essentially due to Smith, \cite{Smith}) states that there exists a family of formulae which behave extremely well with respect to the rank. Let $<$ denote the natural order on $\omega \cup \{ \pm \infty \} $ such that $- \infty < \omega < \infty$ and $<$ restricted to $\omega$ is the natural ordering. We will show that there exists a primitive recursive family of formulae $(\gamma_i)$ such that on this family the rank is locally monotone in $i$.
	
	\begin{lem}[Rank Lemma]
			Let $(M,T) \models \CT^-$ be an arbitrary model and let $r$ be the rank function on this model.
			There exits a primitive recursive family of formulae $(\gamma_i)$ such that for all $a \in M$ if $r(\gamma_a) < + \infty$, then 
			\begin{displaymath}
			r(\gamma_a) < r(\gamma_{a+1}).
			\end{displaymath}
	\end{lem}

The idea of the above lemma is extremely simple: if we consider a natural function measuring how close formulae come to satisfy the main lemma, it turns out that for some carefully chosen family $(\gamma_i)$ this function is locally increasing as long, as it does not hit its maximum. Let us first prove the main lemma using Rank Lemma. Then we shall present the proof of Rank Lemma, since it is admittedly more technical. 
	
	\begin{proof}[Proof of the Main Lemma assuming Rank Lemma]
		Let $(M,T) \models \CT^-$ be an arbitrary nonstandard model. Let $r, \gamma_i$ satisfy the assumptions of the the Rank Lemma. Note that by definition of the rank function a formula $T'(x) :=T\phi(\num{x})$ satisfies the main lemma exactly when $r(\phi) = +\infty.$ We will show that there exists $a \in M$ such that $r(\gamma_a) = + \infty.$
		
		Suppose otherwise. Fix an arbitrary nonstandard $a \in M$. If there is no $b \in M$ with $r(b) = + \infty,$ then by Rank Lemma we have:
		\begin{displaymath}
		r(a) > r(a-1)> r(a-2) > \ldots
		\end{displaymath}  
		which forms an infinite descending chain in the well-order $\{-\infty \} \cup \omega.$ So by contradiction there exists some $a \in M$ such that $r(a) = +\infty.$
	\end{proof}
	
	Before proving Rank Lemma it will be convenient to isolate one easy technical (sub)lemma:
	
	\begin{lem}\label{lem::techn_stand_ext}
		Let $(M,T)\models \CT^-$. Let $\phi(x)$ be a standard arithmetical formula with one free variable and $\gamma(x)$ an arbitrary formula from $\Form(M)$. Suppose that
		\begin{equation}\label{equat::base}
		(M,T)\models \forall t \ \ \phi(t)\equiv T\gamma(t)
		\end{equation}
		Then for an arbitrary standard formula $\delta(x_1,\ldots, x_n)$ of $\mathcal{L}_{\PAP}$ we have
		\[(M,T)\models \forall \bar{t}\ \ \biggl(\delta[\phi](\bar{t})\equiv T\delta[\gamma](\bar{t})\biggr).\]
	\end{lem}
	\begin{proof}
		By induction on the complexity of $\delta$. In the base step we use \eqref{equat::base} (for $\delta = P(x)$) and the fact that $\CT^-\vdash\UTB^-$ (if $\delta$ is atomic $\mathcal{L}_{\PA}$- formula). In the induction step we use compositional axioms of $\CT^-$.
	\end{proof}

	\begin{proof}[Proof of Rank Lemma]
		Let $(M,T) \models \CT^-$ be an arbitrary model. We have to define a primitive recursive sequence of formulae $(\gamma_i)$ satisfying the assumptions of Rank Lemma. Recall that $(\chi_i)$ was an arbitrary primitive recursive enumeration of all arithmetical formulae with at most one free variable and that $(\rho_i)$ was a primitive recursive sequence of arithmetical partial truth predicates such that  $T\rho_a(\num{x})$ satisfies the full uniform disquotation scheme for any nonstandard $a$. Let us define the following sequence of formulae
			\begin{eqnarray*}
				\gamma_0(x) & = & (x=x) \cr
				\gamma_{i+1}(x) & = & \delta[\gamma_i] \longrightarrow \alpha_{i,i}(x) \cr
				\alpha_{i,0}(x) & = & \rho_{2i}(x) \cr
				\alpha_{i,j+1}(x)	& = & \bigl( \forall \bar{t} \ \ \gamma_i(\chi_{i-(j+1)}(\bar{t})) \equiv \chi_{i-(j+1)}(\bar{t}^{\circ}) \bigr) \wedge \alpha_{i,j}(x) \cr
				& & \vee  \neg \bigl( \forall \bar{t} \ \ \gamma_i(\chi_{i-(j+1)}(\bar{t})) \equiv \chi_{i-(j+1)}(\bar{t}^{\circ} )\bigr) \wedge \rho_{2(i-(j+1))}(x), \cr
			\end{eqnarray*}
	where $j+1 \leq i.$ Obviously, this definition may be formalised in $\PA$. We will show that the sequence $(\gamma_i)$ satisfies Rank Lemma. Take an arbitrary $a.$ Suppose that
	\begin{displaymath}
	r(\gamma_a) = - \infty,
	\end{displaymath}
	i.e. 
	\begin{displaymath}
	(M,T) \models \neg T \delta[\gamma_a].
	\end{displaymath}
	Then by definition of $\gamma_{a+1}$ we have
	\begin{displaymath}
	(M,T) \models \forall t \ T \gamma_{a+1}(t).
	\end{displaymath}
	Thus we obtain
	\begin{displaymath}
	(M,T) \models \forall t\ \ \Big( (t=t) \equiv T\gamma_{a+1}(t) \Big)
	\end{displaymath} 
	and consequently (by Lemma \ref{lem::techn_stand_ext} for $\phi(x) := (x=x)$ and the assumption that every standard formula has the property $\delta$)
	\begin{displaymath}
	(M,T) \models T \delta[\gamma_{a+1}].
	\end{displaymath}
	So $r(\gamma_{a+1}) > - \infty.$

	Now, we have to show the key part of the proof. Suppose that 
	\begin{displaymath}
	r(\gamma_a) = n \in \omega.
	\end{displaymath} 	
	This means that for all $i<n$ we have
	\begin{displaymath}
	(M,T) \models \forall \bar{t}\ \ T\gamma_a(\num{\chi_i(\bar{t})}) \equiv \chi_i(\bar{t}^{\circ}), 
	\end{displaymath}
	but the equivalence does not hold for $i = n.$ Now, the formulae $\alpha_{i,j}$ are designed precisely so that they measure when the formula $T \gamma_i$ fails to respect the uniform disquotation scheme and replaces it with a formula of slightly higher rank. In our case we have:
	\begin{displaymath}
	(M,T) \models  \neg T \bigl( \forall \bar{t} \ \ \gamma_a(\num{\chi_{a-(a-n)}}(\num{\bar{t}})) \equiv \chi_{a-(a-n)}(\bar{\num{t^{\circ}}})\bigr)
	\end{displaymath}
	This means that for $i = a$ and $j+1 = a-n$ we get by definition of $\alpha_{i,j}$:
	\begin{displaymath}
	(M,T) \models T \forall x \Big(\alpha_{a,a-n}(x) \equiv \rho_{2n}(x)\Big). 
	\end{displaymath}
	But then the formula $\alpha_{a,a-n}$ satisfies the uniform disquotation scheme for $\chi_0, \ldots, \chi_n$ and therefore by definition of $\alpha_{a,j}$ we have the chain of equivalences:
	\begin{displaymath}
	(M,T) \models T \forall x \Big( \big(\alpha_{a,a-n}(x) \equiv \alpha_{a,a-n+1}(x)\big) \wedge \ldots \wedge \big(\alpha_{a,a-1}(x) \equiv \alpha_{a,a}(x)\big)  \Big).
	\end{displaymath}
	So by definition of  $\gamma_{a+1}$  it turns out that
	\begin{displaymath}
	(M,T) \models T \forall x \Big(\gamma_{a+1}(x) \equiv \rho_{2n}(x)\Big).
	\end{displaymath}
	The formula $\rho_{2n}$ is standard and all standard formulae have the property $\delta$. So we have
	\begin{displaymath}
	(M,T) \models T\delta[\rho_{2n}].
	\end{displaymath}
	This together with the fact that the formula $\rho_{2n}$ satisfies the uniform disquotation scheme for $\chi_0, \ldots, \chi_{n}$ means that $r(\gamma_{a+1}) > r(\gamma_a).$
	\end{proof}

	We are almost ready to prove the main theorem. Unfortunately, we still have to consider one very technical issue. In our proof we would like to use at some point a property of \df{generalised commutativity} of compositional truth predicates. I.e. for a compositional truth predicate $T'$ we would like to have the following equivalence for all standard formulae $\phi$ from $\mathcal{L}_{\PAP}$ (i.e. formulae from the arithmetical language possibly enriched by one fresh unary predicate $P(v)$) and (possibly nonstandard) arithmetical formulae $\eta$ 
	\begin{displaymath}
	T' \phi[\eta] \equiv \phi[T'\eta].
	\end{displaymath}	
	Unfortunately, this clean equivalence is not literally true. Namely, we have, e.g.
	\begin{displaymath}
	T' \Big(\exists x \ \ x=x \Big) \equiv \exists t \ T'(t=t)
	\end{displaymath}
	and the last sentence is simply not the same as
	\begin{displaymath}
	\exists x \ T'(x=x),
	\end{displaymath}
	since $x=x$ is a formulae with a free variable rather than a sentence. Thus in order to use some form of the generalised compositionality we have to make some technical amendments. There is a couple of ways this might be done. We decide for the one which we believe makes our considerations most perspicuous.
	
	\begin{defin}
		Let $\phi(x_1, \ldots, x_n)$ be a standard formula from the language $\mathcal{L}_{\PAP}$. We say that $\phi$ is \df{semirelational} if it has no subformula of the form $P(t)$, where $t$ is some term other than a first order variable. 
	\end{defin}
	
	In what follows we will restrict our attention to semirelational formulae. We can do this without loss of generality thanks to the next lemma.
	
	\begin{lem} [Semirelational normal form]
		Let $\phi(x_1, \ldots, x_n)$ be an arbitrary formula from $\mathcal{L}_{\PAP}$. Then there exists a semirelational formula $\phi'(x_1, \ldots, x_n)$ in the same language such that
		\begin{displaymath}
		 \models \forall P \ \forall x_1, \ldots, x_n \big(\phi(x_1,\ldots, x_n) \equiv \phi'(x_1,\ldots, x_n)\big).
		\end{displaymath}
	\end{lem}
	
	The proof of the above lemma is completely straightforward. We replace all the expressions of the form $P(t)$ with $\exists v \ ( v=t \wedge P(v))$ with $v$ chosen so that we can avoid clashes of variables. 
	
	\begin{defin}
		Let $(M,T)$ be any model of $\CT^-$ and let $\Phi$ be any subset of $\Form (M).$ We say that a formula $T'(x)$ is \df{extensional} for formulae in $\Phi$ if for all $\phi \in \Phi$ we have
		\begin{displaymath}
		\bar{s}^{\circ} = \bar{t}^{\circ} \rightarrow \big(T'\phi(\bar{s}) \equiv T'\phi(\bar{t}) \big).
		\end{displaymath}
	\end{defin}
	
	Extensionality is a very important property of a truth predicate which may fail in the absence of induction. One of its consequences is the following equivalence
	\[\forall \phi\ \ \forall t T\phi(t)\equiv \forall x T\phi(\num{x})\]
	which does not trivialise if the language of $\PA$ is assumed to have function symbols for addition and multiplication. One may show that the above equivalence is independent from the axioms of $\CT^-$. That's why we will additionally demand this property from our partial truth predicates.

	\begin{lem}[Generalised commutativity] \label{lem_generalised_com}
		Let$(M,T)$ be any model of $\CT^-$. Suppose that $T'(x)$ satisfies compositional axioms and is extensional for all formulae of the form $\phi[\xi],$ where $\phi$ is a standard semirelational formula from the language $\mathcal{L}_{\PAP}$ and $\xi \in \Form(M)$ is an arbitrary arithmetical formula, possibly nonstandard, with at most one free variable. Let us define 
		\begin{displaymath}
		\hat{\xi}(x) := T'\xi(\num{x}).
		\end{displaymath}
		Then the following equivalence holds
		\begin{displaymath}
		(M,T) \models \forall x_1, \ldots, x_n \Big( T' \phi[\xi](\num{x_1}, \ldots, \num{x_n}) \equiv \phi[\hat{\xi}](x_1, \ldots, x_n) \Big).
		\end{displaymath}
	\end{lem}
	
	\begin{proof}
		We check the claim by induction on the complexity of $\phi.$
	\end{proof}
		
	Now we are ready to prove our theorem. As a matter of fact, we shall obtain slightly stronger result. The predicate $T$ we are going to construct is going to display an additional property that 
	$$(M,T) \models \UTB $$ 
	is recursively saturated \emph{as a model of \UTB}.
	
	\begin{proof}
		Let $(M,T) \models \CT^-$ and let  $\tilde{\theta}(y)$ be a formula '$P(x)$ is a compositional extensional truth predicate for formulae $<y$' i. e. a conjunction of the following formulae:
		
		\begin{enumerate}
			\item $\forall \phi < y  \ \ P(\neg \phi) \equiv \neg P(\phi).$
			\item $\forall \phi, \psi < y \ \ P(\phi \odot \psi) \equiv P(\phi) \odot P(\psi).$
			\item $\forall \phi < y \ \ P(Q x \phi) \equiv Qt \ P(\phi(t)),$		
 			\item $\forall \phi< y \forall s,t \ \ \Big( s^{\circ} = t^{\circ} \rightarrow (P\phi(s) \equiv P\phi(t)\Big)$, 
		\end{enumerate}
	where $\odot \in \{\wedge, \vee \}, Q \in \{\forall, \exists \}.$
	\begin{comment}
	Note that the compositionality  in the sense of the  formula $\tilde{\theta}$ is not quite the same as the one encapsulated in the axioms of $\CT^-$. Namely: the quantifier axioms say that e.g. a universal formula is true iff it is true of every \emph{numeral}, rather than every \emph{term} and the quantifier axioms in two formulations do not even seem comparable. This technical distinction will play a crucial role in our proof. 
	\end{comment}
		Let $\theta$ be a sentence '$\tilde{\theta}$ is inductive' i.e.
		
		$$ \forall x \Bigl( \tilde{\theta}(x) \rightarrow \tilde{\theta}(x+1) \Bigr) \longrightarrow 
		\Bigl( \tilde{\theta}(0) \rightarrow \forall x \tilde{\theta}(x) \Bigr).$$
		
		Let now $(ind_k)$ be some recursive enumeration of the  instances of the induction scheme for semirelational formulae in the arithmetical language with the additional predicate $P(x)$. For an arbitrary formula $\phi$ with at most one free variable let $Ind_k (\phi)$ be a conjunction of the first $k$ instances of the form $ind_j[\phi]$ with parentheses grouped to the right. The reader should not be confused with the fact, that we used a predicate $P$ to define $Ind_k(\phi).$ It does not occur in our formula anymore.
		
		Let $\tilde{\zeta}(x)$ be defined as
		
		$$\forall y< x \ \ P(Ind_y(\rho_y)).$$
		
		So it is a formula saying ''$P$ holds of the formula 'the formula $\rho_y$ satisfies first $y$ instances of the induction scheme' for all $y$ smaller than $x$''. In other words ''the formulae $Ind_y(\rho_y)$ are true for all $y<x$'', since the intended meaning of the predicate $P(x)$ is the truth predicate. Let $\zeta$ be defined in an analogous fashion to $\theta$ i.e.
		
		$$\zeta = \forall y \Bigl( \ \tilde{\zeta}(y) \rightarrow \tilde{\zeta}(y+1) \Bigr) 
							\rightarrow \Bigl( \tilde{\zeta}(0) \rightarrow \forall y \ \ \tilde{\zeta}(y) \Bigr).$$
							
		Let finally 
		
		$$\delta = \zeta \wedge \theta.$$
		
		Observe that every standard formula $\phi$ has the property $\delta,$ since $\delta[\phi]$ is simply an instance of the induction scheme. So by our lemma there is a formula $T'(x)$ such that
		
		$$(M,T) \models \delta[T'] $$  
		and $T'$ satisfies uniform disquotation scheme. Since it satisfies the scheme, it is compositional for standard formulae, i.e. for all $k \in \omega$
		
		$$(M,T) \models \tilde{\theta}[T'](k)$$
		So, by overspill ($T'$ satisfies $\theta$!) we have 
		$$(M,T) \models \tilde{\theta}[T'](c)$$ for some $c > \omega.$
		
		Now, since $\rho_k$ for $k \in \omega$ are standard formulae, they satisfy full induction scheme. In particular 
		$$(M,T) \models \tilde{\zeta}[T'](k).$$
So applying overspill ($T'$ satisfies $\zeta$!) once more we get some nonstandard $d$ such that for all $e<d$ we have
		$$(M,T) \models \tilde{\zeta}[T'](e).$$
Let us fix a nonstandard $e<d,c$ which is much smaller than $c$, so that it satisfies all the inequalities needed in the further part of the proof. Let us list these inequalities right now, although it is not essential to understand their role in advance. We also assume that we use a coding under which if $\phi$ is a subformula of $\psi$,  then the code of $\phi$ is smaller than the code of $\psi$. The reader is  warned that if we assume that $e$ is nonstandard then actually all the three points follow from the second one:
\begin{enumerate}
	\item $\rho_e < c.$
	\item $Ind_e(\rho_e) <c.$
	\item For all standard formulae $\eta$ we have $\eta[\rho_e] < c.$
\end{enumerate} 

Let $T''(x)$ be defined as $T'(\rho_e(\num{x}))$. We claim that $$(M,T'') \models \UTB. $$ 

Since $e$ is much smaller than $c$, we may assume that
\begin{displaymath}
\rho_e < c,
\end{displaymath}
 so that our predicate is compositional and we may show that it satisfies the uniform disquotation scheme in exactly the same way as in the case of  $\CT^-.$  It is enough to show that it satisfies the full induction scheme, that is we have to show that
\begin{equation} \tag{*} \label{equat_ind}
(M,T) \models \forall x \Big(\phi[T''](x) \rightarrow \phi[T''](Sx)\Big) \rightarrow \Big(\phi[T''](0) \rightarrow \forall x \ \phi[T''](x) \Big)
\end{equation}
for an arbitrary standard semirelational $\phi$ from $\mathcal{L}_{\PAP}.$ Note that for some $k \in \omega$ the following equality holds:
\begin{displaymath}
 ind_k[\rho_e] =\forall x \Big(\phi[\rho_e](x) \rightarrow \phi[\rho_e](Sx)\Big) \rightarrow \Big(\phi[\rho_e](0) \rightarrow \forall x \ \phi[\rho_e](x) \Big).
\end{displaymath}
 By assumption we have
$$(M,T) \models T' Ind_e (\rho_e).$$
Now, $T'$ is compositional for formulae $<c$ and we assumed that $Ind_e(\rho_e) < c$ and since $k \in \omega$, the formula $ind_k$ is located in the formula $Ind_d$ on finite syntactic depth. Thus we get:
\begin{displaymath}
(M,T) \models T' \biggl(\forall x \Big(\phi[\rho_e](x) \rightarrow \phi[\rho_e](Sx)\Big) \rightarrow \Big(\phi[\rho_e](0) \rightarrow \forall x \ \phi[\rho_e](x) \Big)\biggr).
\end{displaymath}
By compositionality of $T'$ and Lemma \ref{lem_generalised_com} for $\hat{\xi} = T''$ this implies \eqref{equat_ind}.

\begin{comment}
This may be expanded as
\begin{displaymath}
(M,T) \models T' \biggl( \forall x \Bigl( \phi[\rho_e](x) \rightarrow \phi[\rho_e](Sx) \Bigr) \longrightarrow \Bigl( \phi[\rho_e](0) \rightarrow \forall x \  \phi[\rho_e](x)  \Bigr)\biggr).
\end{displaymath}
For some standard arithmetical $\phi$. This implies, again by compositionality
\begin{displaymath}
(M,T) \models  \biggl( \forall x \Bigl( T'\phi[\rho_e](\num{x}) \rightarrow T'\phi[\rho_e](\num{Sx}) \Bigr) \longrightarrow \Bigl( T'\phi[\rho_e](\num{0}) \rightarrow \forall x  \ T'\phi[\rho_e](\num{x})  \Bigr)\biggr),
\end{displaymath}
which yields by the Generalised Compositionality Lemma:
\begin{displaymath}
(M,T) \models  \biggl( \forall x \Bigl( \phi[T''](x) \rightarrow \phi[T''](Sx) \Bigr) \longrightarrow \Bigl( \phi[T''](0) \rightarrow \forall x \  \phi[T''](x)  \Bigr)\biggr)
\end{displaymath}
and this means precisely that
\begin{displaymath}
(M,T) \models ind_k[T''].
\end{displaymath}
So indeed $(M,T'\rho_d(x))$ is a model of $UTB.$ 
\end{comment}
	
	\end{proof}	
		
		An inspection of the proof shows that the $\UTB$ predicate we have defined is of the form 
		$$ T \gamma_a (\num{\rho_b(\num{x})}),$$ for some nonstandard $a,b.$ Thus by Theorem 3.1 of \cite{Smith} the model we defined is recursively saturated. Also by inspection of proof, we obtain the following corollary:
		
		\begin{stw}
		For any $(M,T) \models \UTB,$ there exists $T' \subset M$ such that $(M,T') \models \UTB$ is recursively saturated as a model of $\UTB.$
		\end{stw}

		\begin{proof}
		Note that the proof of our main theorem may be repeated with a weaker assumption that $(M,T)$ is a model of \emph{partial} compositional truth predicate, i. e. truth predicate compositional for formulae $ \leq c$ for some nonstandard $c \in M.$ But if $(M,T) \models \UTB$, then by overspill there exists some $b \in M$ such that $T\upharpoonright_b$ is partial compositional, where $T\upharpoonright_b$ is a restriction of the predicate $T$ to formulae smaller than $b.$ 
		\end{proof}

	There is also one corollary implicit in the proof of the theorem, that seems to be worth stating explicitly.
		
		\begin{stw}
		Let $(M,T)\models \CT^-.$ Then there exists a truth predicate for $M$ satisfying $\UTB$ which is definable in $(M,T)$ with parameters. 
		\end{stw}
				
Note that by a standard overspill argument a model carries a $\UTB$ class precisely when it carries a fully inductive partial compositional truth class. As we have seen above, a model $M \models \PA$ with a partial compositional truth class carries also a $\UTB$ class. Thus we may conclude that a model $M \models \PA$ carries a partial compositional truth class if and only if it carries an \emph{inductive} partial compositional truth class, thus we may get full induction in this setting somewhat for free. We did not check the details, but we believe, that the results carry over to the satisfaction predicate setting. Thus a corollary in the more classical language of satisfaction predicates would be as follows: an arbitrary model $M \models \PA$ has a partial satisfaction class if and only if it has a partial inductive satisfaction class.		
		
Let us close this section with some remarks of the possible strategy of proving that
	$$\mathfrak{CT}^- \neq \mathfrak{UTB}.$$

Models of $\UTB$ display a good deal of structural properties which they possibly do not share with the models of $\CT^-.$ Namely:
	
	\begin{enumerate}
		\item Every model in $\mathfrak{UTB}$ has an elementary end-extension in $\mathfrak{UTB}.$
		\item The class $\mathfrak{UTB}$ is closed under cofinal extensions.
		\item Every recursively saturated model of countable cofinality is in $\mathfrak{UTB}.$
	\end{enumerate}
		
We conjecture that for $\mathfrak{CT}^-$ all these properties fail, although we have not managed to show it, nor have we any clue how to prove it.

\end{document}